\documentclass[12pt, reqno]{amsart}

\usepackage{amsthm, amsmath, amssymb,booktabs,xcolor,graphicx,stackrel}

\usepackage[hidelinks]{hyperref}

\usepackage[color,final]{showkeys} 

\colorlet{refkey}{orange!20}
\colorlet{labelkey}{blue!30}

\newtheorem{theorem}{Theorem}[section]

\newtheorem{conjecture}[theorem]{Conjecture}
\newtheorem*{question*}{Question}
\newtheorem*{theorem*}{Theorem}

\theoremstyle{definition}
\newtheorem{definition}[theorem]{Definition}

\newtheorem{question}[theorem]{Question}
\newtheorem*{definition*}{Definition}

\theoremstyle{remark}

\newcommand{\ang}[1]{\left\langle #1 \right\rangle}

\newcommand{\paren}[1]{\left( #1 \right)}

\newcommand{\wh}{\widehat}

\newcommand{\ol}{\overline}

\newcommand{\EE}{\mathbb{E}}
\newcommand{\E}{\mathsf{E}}

\newcommand{\FF}{\mathbb{F}}
\newcommand{\CC}{\mathbb{C}}

\newcommand{\RR}{\mathbb{R}}

\newcommand{\bx}{\mathbf{x}}

\author[Fox]{Jacob Fox}
\author[Pham]{Huy Tuan Pham}
\email{\{jacobfox,huypham\}@stanford.edu}
\address{Department of Mathematics, Stanford University, Stanford, CA 94305.}

\author[Zhao]{Yufei Zhao}
\email{yufeiz@mit.edu}
\address{Department of Mathematics, Massachusetts Institute of Technology, Cambridge, MA 02139.}

\thanks{Fox was supported by a Packard Fellowship and by NSF grant DMS-1855635. Zhao was supported by NSF Award DMS-1764176, the MIT Solomon Buchsbaum Fund, and a Sloan Research Fellowship.}

\title{Common and Sidorenko linear equations}

\begin{document}

\begin{abstract} 
A linear equation with coefficients in $\mathbb{F}_q$ is {\it common} if the number of monochromatic solutions in any two-coloring of $\mathbb{F}_q^n$ is asymptotically (as $n \to \infty$) at least the number expected in a random two-coloring. The linear equation is {\it Sidorenko} if the number of solutions in any dense subset of $\mathbb{F}_q^n$ is asymptotically at least the number expected in a random set of the same density. 

In this paper, we characterize those linear equations which are common, and those which are Sidorenko. The main novelty is a construction based on choosing random Fourier coefficients that shows that certain linear equations do not have these properties. This solves problems posed in a paper of Saad and Wolf. 
\end{abstract}

\maketitle

\section{Introduction} \label{sec:introduction}

Sidorenko's conjecture \cite{Si93}  (also conjectured earlier in a stronger form by Erd\H{o}s and Simonovits \cite{Simon}) is a major open problem in extremal graph theory. We say that a bipartite graph $H$ is \emph{Sidorenko} if the density of copies of $H$ in a graph with fixed edge density is asymptotically minimized by the random graph with the same edge density. Sidorenko's conjecture says that all bipartite graphs $H$ are Sidorenko. Many graphs are known to have the Sidorenko property, including bipartite graphs with a vertex complete to the other part, see \cite{CoFoSu,CKLL,CoLee2,CoLee,Hat,KLL,LiSz,Sz}. A coloring variant of Sidorenko's conjecture, the Burr-Rosta conjecture \cite{BR} (extending an earlier conjecture of Erd\H{o}s \cite{Erd}), claims that the density of monochromatic copies of any fixed graph $H$ in any coloring of $K_n$ is asymptotically minimized by the random coloring. While the Burr-Rosta conjecture was disproved by Thomason \cite{T} and Sidorenko \cite{Sid1}, many graphs $H$ are known to satisfy the Burr-Rosta conjecture; such graphs are called {\it common graphs}. 

Saad and Wolf \cite{SaWo} explored various analogues of Sidorenko and common graphs in the arithmetic setting. In this setting, we consider the minimum density of solutions to a system of linear equations in a subset of given density in a finite abelian group, or the minimum density of monochromatic solutions to a system of linear equations in a two-coloring of the abelian group. In particular, we consider the setting of a fixed linear system $L$ with coefficients in some finite field $\FF_q$. We say that a system of linear equations $L=0$ is \emph{common} if the density of monochromatic solutions to this system in any two-coloring of $\FF_q^n$ is asymptotically at least what we expect from a random coloring. Likewise, we say that a linear system $L=0$ is \emph{Sidorenko} if the density of solutions to this system in any dense subset of $\FF_q^n$ is asymptotically at least what we expect from a random set with the same density. The formal definitions for single linear homogeneous equations are given below. While the above arithmetic problems do not directly correspond with graph problems, they share many common features. We refer the reader to \cite{SaWo} for more details.

\begin{definition} \label{def:sid-common}
Given a linear form $L(x_1, \dots, x_k) = a_1 x_1 + \cdots + a_k x_k$ with $a_1, \dots, a_k \in \FF_q$, we say that the equation $L = 0$ is \emph{Sidorenko} if for every $n$ and every $A \subseteq \FF_q^n$, the number of solutions to $L(x_1, \dots, x_k) = 0$ with $x_1, \dots, x_k \in A$ is at least $|A|^k/q^n$. We say that the equation $L = 0$ is \emph{common} if for every $n$ and every coloring of $\FF_q^n$ with two colors, the number of monochromatic solutions $(x_1, \dots, x_k) \in (\FF_q^n)^k$ to $L(x_1, \dots, x_k) = 0$ is at least $2^{1-k}q^{n(k-1)}$.
\end{definition}

Note that if we choose a random subset $S$ of $\FF_q^n$ by picking each element independently with probability $|A|/q^n$, then the expected number of solutions to $L(x_1,\dots,x_k)=0$ with $x_1,\dots,x_k\in S$ is $q^{n(k-1)}(|A|^k/q^{nk} + o_{n \to \infty}(1))$ where the $o_{n \to \infty}(1)$ term accounts for the density of solutions of $L(x_1,\dots,x_k)=0$ with $x_i=x_j$ for some $i\ne j$. Here, we denote by $o_{n \to \infty}(1)$ a term which tends to $0$ uniformly in $A$ as $n\to \infty$. Similarly, in a random two-coloring of $\FF_q^n$ where the color of each element is chosen independently and uniformly, the expected number of monochromatic solutions to $L(x_1,\dots,x_k)=0$ is $q^{n(k-1)}(2^{1-k}+o_{n \to \infty}(1))$. Moreover, by a simple argument given at the beginning of Section \ref{sec:main}, we can show that if for every $A\subseteq \FF_q^n$, the number of solutions to $L(x_1,\dots,x_k)=0$ with $x_1,\dots,x_k\in A$ is at least $q^{n(k-1)}(|A|^k/q^{nk}+o_{n \to \infty}(1))$, then in fact $L=0$ is Sidorenko. Similarly, if the number of monochromatic solutions to $L(x_1,\dots,x_k)=0$ is at least $q^{n(k-1)}(2^{1-k}+o_{n \to \infty}(1))$, then $L=0$ is common. 

Cameron, Cilleruelo, and Serra \cite{CCS} proved using a cancellation argument that, in an abelian group $G$,
given a linear equation $L = 0$ with an \emph{odd} number of variables and whose coefficients are coprime to $|G|$,
the number of monochromatic solutions to $L=0$ in any 2-coloring of $G$ depends only on the size of the color classes.
It easily follows that the linear equation is common in this case. Saad and Wolf  \cite[Conjecture 5.2]{SaWo} made the following conjecture for linear equations with an \emph{even} number of variables.

\begin{conjecture}[\cite{SaWo}]
Let $k\ge 2$ be an integer. A linear equation of the form $$a_1x_1+\cdots+a_{2k}x_{2k}=0$$ is common in $\FF_q^n$ if and only if we can partition $\{a_1,...,a_{2k}\}$ into $k$ pairs, each summing to $0$. 
\end{conjecture}

Saad and Wolf \cite{SaWo} noted that the ``if'' direction of their conjecture follows easily from an application of the Cauchy-Schwarz inequality. 

The following question was attributed to Alon \cite[Question 4.1]{SaWo}.

\begin{question}\label{q:alon}
Is it true that adding sufficiently many free variables makes any linear system not common?
\end{question}

In this paper, we characterize all linear homogeneous equations (i.e., systems with one equation) that are common as well as those that are Sidorenko, resolving the above conjecture and question.
In the theorem below, part (b) is new, which we prove by constructing a function with randomly chosen Fourier coefficients. Parts (a) and (c) are previously known~\cite{SaWo}, though we include their proofs for completeness.

\begin{theorem} \label{thm:main}
    Let $L(x_1, \dots, x_k) = a_1 x_1 + \cdots + a_k x_k$ be a linear form with $a_1, \dots, a_k \in \FF_q \setminus\{0\}$. 
	\begin{enumerate}
	\item[(a)] \label{match} If $a_1, \dots, a_{k}$ can be partitioned into pairs each summing to zero, then the equation $L = 0$ is Sidorenko and common.
	\item[(b)] \label{not-match} If $k$ is even and $a_1, \dots, a_{k}$ cannot be partitioned into pairs each summing to zero, then the equation $L = 0$  is not common and not Sidorenko.
	\item[(c)] \label{odd} If $k$ is odd, then the equation $L = 0$  is common but not Sidorenko.
	\end{enumerate}
\end{theorem}

Adding a free variable to a linear equation is the same as adding a variable with coefficient $0$ to the equation. Hence, the density of solutions to the new equation (with $\ell$ coefficients being $0$) is simply the density of solutions to the original equation (with only nonzero coefficients), multiplied by the set density to the power $\ell$. The following theorem answers Question~\ref{q:alon}. 

\begin{theorem} \label{thm:free-var}
    Let $L(x_1, \dots, x_k) = a_1 x_1 + \cdots + a_k x_k$ be a linear form with $a_1, \dots, a_k \in \FF_q \setminus\{0\}$. 
    Let $L'$ be the linear form obtained by adding $\ell \ge 1$ free variables to $L$ (equivalently, $L'(x_1, \dots, x_{k+\ell}) = a_1 x_1 + \cdots + a_k x_k + 0 x_{k+1} + \cdots + 0 x_{k + \ell}$). If $a_1, \dots, a_{k}$ can be partitioned into pairs each summing to zero, the equation $L' = 0$  is Sidorenko and common, and otherwise $L' = 0$  is not common and not Sidorenko.
\end{theorem}

It is natural to study the same problems for inhomogeneous linear equations. 

\begin{definition} \label{def:inhom}
Let $L(x_1, \dots, x_k) = a_1 x_1 + \cdots + a_k x_k$ be a linear form with $a_1, \dots, a_k \in \FF_q \setminus\{0\}$. We say that the linear form $L$ is \emph{inhomogeneous-Sidorenko} if for every $n$, every nonzero $b\in \FF_q^n$, and every $A \subseteq \FF_q^n$, the number of solutions to $L(x_1, \dots, x_k) = b$ with $x_1, \dots, x_k \in A$ is at least $|A|^k/q^n$. We say that the linear form $L$ is \emph{inhomogeneous-common} if for every $n$, every nonzero $b\in \FF_q^n$, and every coloring of $\FF_q^n$ with two colors, the number of monochromatic solutions $(x_1, \dots, x_k) \in (\FF_q^n)^k$ to $L(x_1, \dots, x_k) = b$ is at least $2^{1-k}q^{n(k-1)}$. 
\end{definition}

We remark that the choice of nonzero $b$ above is inconsequential, since for any $n$, if the properties in Definition \ref{def:inhom} are satisfied for any nonzero $b\in \FF_q^n$, then they are satisfied for every nonzero $b \in \FF_q^n$. Indeed, for any nonzero $b,b' \in \FF_q^n$, we can find an invertible linear transformation $C$ such that $Cb=b'$. Then $a_1x_1+\cdots+a_kx_k=b$ if and only if  $a_1Cx_1+\cdots+a_kCx_k=b'$, so the solutions to $L=b$ in $A$ and the solutions to $L=b'$ in $CA$ are in one-to-one correspondence via the invertible transformation $C$. The next result  gives a simple characterization of the inhomogeneous-common linear forms, and shows that no linear form is inhomogeneous-Sidorenko.

\begin{theorem} \label{thm:inhom}
    Let $L(x_1, \dots, x_k) = a_1 x_1 + \cdots + a_k x_k$ be a linear form with $a_1, \dots, a_k \in \FF_q \setminus\{0\}$. 
    Let $L'$ be the linear form obtained by adding $\ell \ge 0$ free variables to $L$.
    
    Then $L'$ is never inhomogeneous-Sidorenko, and $L'$ is inhomogeneous-common if and only if $k$ is odd and $\ell = 0$.
\end{theorem}

We note that Leo Versteegen~\cite{V} has generalized our results to arbitrary abelian groups.

\section{Proofs}\label{sec:main}

We will use two different notations for expectation: $\E$ and $\EE$.
First, the symbol $\E$ denotes averaging. For example, for a function $f \colon \FF_q^n \to \RR$, we write $\E f = \E_\bx f(\bx) = \E_{\bx \in \FF_q^n} f(\bx)$ to denote the average value of $f(\bx)$ as $\bx$ varies uniformly over its domain (we omit the subscripts when there is no confusion). 
Second, as we will be constructing a random set $A$, we write $\EE_A$ to denote the expectation over this probability distribution.

We note that the Sidorenko property is equivalent to its functional version, where we replace the subset $A \subseteq \FF_q^n$ by a function $f:\FF_q^n\to [0,1]$. Indeed, given a linear form $L(x_1, \dots, x_k) = a_1 x_1 + \cdots + a_k x_k$ and a function $f \colon \FF_q^n \to [0,1]$, we write
\[
\Lambda_{L=0}(f) := \E_{\bx = (x_1, \dots, x_k) \in (\FF_q^n)^k : L(\bx) = 0} \left[f(x_1) \cdots f(x_{k})\right].
\]
Then the equation $L = 0$ is Sidorenko if and only if for every $n$ and every $f \colon \FF_q^n \to [0,1]$,
\begin{equation}\label{eq:functional}
\Lambda_{L=0}(f) \ge (\E f)^{k}.
\end{equation}
Indeed, being Sidorenko in the sense of Definition~\ref{def:sid-common} is equivalent to \eqref{eq:functional} for all $f$ of the form $1_A$ with $A \subseteq \FF_q^n$. Conversely, suppose \eqref{eq:functional} fails for some $f \colon \FF_q^n \to [0,1]$.
We can extend $f$ to a function on $\FF_q^{n+n'}$ by forgetting the $n'$ new coordinates when evaluating $f$. We can now sample a random subset $A \subseteq \FF_q^{n+n'}$ by independently including each $x \in \FF_q^{n+n'}$ with probability $f(x)$. We have $\EE_A[ \E 1_A] = \E[f]$, and $\EE_A \Lambda_{L=0}(1_A) = \Lambda_{L=0}(f) + o_{n'\to\infty}(1)$, where the $o_{n'\to\infty}(1)$ term accounts for the proportion of $\bx \in (\FF_q^{n+n'})^k$ with $L(\bx) = 0$ and not all $k$ coordinates distinct, which goes to $0$ as $n' \to \infty$. Since $\EE_A[(\E 1_A)^k] \ge (\EE_A[\E 1_A])^k = \E[f]^k$ by the convexity of $t\mapsto t^k$, it follows that if \eqref{eq:functional} fails for some $f$, then for $n'$ large enough, there exists $A \subseteq \FF_q^{n+n'}$ that such that \eqref{eq:functional} fails for $f = 1_A$. Thus the set version and the functional formulations of the property of being Sidorenko are equivalent. The same argument shows that if $\Lambda_{L=0}(A) \ge (|A|/q^{n})^k + o_{n \to \infty}(1)$ for every $A$, where the $o_{n \to \infty}(1)$ term tends to $0$ as $n\to \infty$ uniformly in $A$, then in fact $\Lambda_{L=0}(A) \ge (|A|/q^n)^k$. 

Likewise, the property of being common also has an equivalent functional formulation: given a $k$-variable linear form $L$ over $\FF_q$, the equation $L = 0$ is common if and only if for every $f \colon \FF_q^n \to [0,1]$, 
\[
\Lambda_{L=0}(f) + \Lambda_{L=0}(1-f) \ge 2^{1-k}.
\]

We denote the Fourier transform of a function $f \colon \FF_q^n \to \CC$ by
\[
\wh f(r) = \E_{x \in \FF_q^n} f(x) \ol{r(x)}, \qquad r \in \wh{\FF_q^n}
\]
where $\wh{\FF_q^n}$ is the group of characters of $\FF_q^n$, i.e., homomorphisms $r\colon \FF_q^n \to \CC^\times$.
The groups $\wh{\FF_q^n}$ and $\FF_q^n$ are isomorphic by associating $y \in \FF_q^n$ with the character $\gamma_y \in \wh{\FF_q^n}$ defined by $\gamma_y(x) = \exp(2\pi i \operatorname{tr} \ang{x,y}/p)$ where $p$ is the characteristic of $\FF_q$ and $\operatorname{tr} \colon \FF_q \to \FF_p$ is the standard trace map. 
We write the dual group $\wh \FF_q$ additively, so that, e.g., $0 \in \wh {\FF_q^n}$ is the constant-$1$ function $\gamma_0$, and $a \gamma_y = \gamma_{ay} \in \wh {\FF_q^n}$ for any $a \in \FF_q$ and $y \in \FF_q^n$.

Recall the following identity which relates a twisted convolution with the Fourier transform: given $L(x_1, \dots, x_k) = a_1 x_1 + \dots + a_k x_k$,
\begin{equation} \label{eq:fourier-L}
\Lambda_{L=0}(f) = \sum_{r \in \wh{\FF_q^n}} \wh f(a_1 r) \cdots \wh f(a_k r).
\end{equation}
Since $\wh f (0) = \E f$ and $\wh{(1-f)}(r) = -\wh f (r)$ for all $r \ne 0$, we have
\begin{equation} \label{eq:common-fourier}
\Lambda_{L=0}(f) + \Lambda_{L=0}(1-f) = (\E f)^k + (1 - \E f)^k + (1+(-1)^k)\sum_{r \in \wh{\FF_q^n} \setminus \{0\}} \wh f(a_1 r) \cdots \wh f(a_k r) .
\end{equation}

Next, we give the proof of Theorem~\ref{thm:main}. We remark that part (a) of Theorem~\ref{thm:main} has already been proven in \cite{SaWo}.

\begin{proof}[Proof of Theorem~\ref{thm:main}]
(a) Write $m = k/2$ and $L(x_1, x'_1, \dots, x_m, x'_m)  = a_1(x_1-x'_1) + \cdots + a_m(x_m - x'_m)$. For $f \colon \FF_q^n \to [0,1]$, we have
\begin{align*}
\Lambda_{L=0}(f) 
&= \E_z \left( \E_{x_1, \dots, x_m : z = a_1 x_1 + \cdots + a_m x_m} f(x_1) \cdots f(x_m)\right)^2 
\\
&\ge \left(\E_z \E_{x_1, \dots, x_m : z = a_1 x_1 + \cdots + a_m x_m} f(x_1) \cdots f(x_m)\right)^2 
\\
&= \left(\E_{x_1, \dots, x_m} f(x_1) \cdots f(x_m)\right)^2
\\&= (\E f)^{2m}.
\end{align*}
Thus $L=0$ is Sidorenko, which then implies that it must be common, since $\Lambda_{L=0}(f) + \Lambda_{L=0}(1-f) \ge (\E f)^k + (1 - \E f)^k \ge 2^{1-k}$ by the convexity of $t\mapsto t^k$.

\medskip

\noindent (b) It suffices to show that $L=0$ is not common, since as we just noted, every Sidorenko equation is automatically common. It suffices to show that there exists $f \colon \FF_q \to [0,1]$ such that 
\[
\Lambda_{L=0}(f) + \Lambda_{L=0}(1-f) < 2^{1-k}.
\]
In particular, the Sidorenko condition fails for $n=1$ (also see the comments at the beginning of this section).

Since $k$ is even, using \eqref{eq:common-fourier}, it remains to exhibit some $f \colon \FF_q \to [0,1]$ such that $\E f = 1/2$ and such that 
$\sum_{r \in \wh{\FF_q^n} \setminus \{0\}} \wh f(a_1 r) \cdots \wh f(a_k r) < 0$. We shall do it by choosing random values for $\wh f$.

For each $r \in \wh \FF_q \setminus \{0\}$, let $\xi_r$ be a random unit complex number. If $q$ is odd, we choose the $\xi_r$'s subject to $\xi_r = \ol{\xi_{-r}}$ but i.i.d.\ uniform otherwise. If $q$ is even, we choose each $\xi_r \in \{-1,1\}$ uniformly i.i.d.
Construct a random function $f \colon \FF_q \to [0,1]$ by setting $\wh f(0) = 1/2$ and $\wh f(r) = \xi_r/(2q)$ for every $r \in \wh \FF_q \setminus \{0\}$. By the inverse transform $f(x) = \sum_{r \in \wh{\FF_q}} \wh f(r) r(x)$, we see that such $f$ takes real values in $[0,1]$.

Then, for every $r \in \wh \FF_q\setminus\{0\}$, the product $\wh f(a_1 r) \cdots \wh f(a_k r)$ has expectation zero over this random $f$, since $a_1, \dots, a_k$ cannot be partitioned into ``canceling'' pairs each summing to zero. Thus, the expectation of \eqref{eq:common-fourier} over this random $f$ is equal to $2^{1-k}$. 

Observe that for a sufficiently small choice of $\epsilon > 0$, 
if we have $|\xi_r - 1| < \epsilon$ for all $r \in \wh{\FF_q^n} \setminus \{0\}$, then $\Re(\xi_{a_1 r} \cdots \xi_{a_k r}) > \frac{1}{2}$ for all $r\ne 0$. With positive probability, we have that $|\xi_r - 1| < \epsilon$ for all $r \in \wh{\FF_q} \setminus \{0\}$. Under this event, noting that the value of \eqref{eq:common-fourier} is a real number, we have
\[
\sum_{r \in \wh{\FF_q} \setminus \{0\}} \wh f(a_1 r) \cdots \wh f(a_k r)
=
\sum_{r \in \wh{\FF_q} \setminus \{0\}} \frac{\xi_{a_1 r} \cdots \xi_{a_k r}}{(2q)^k}
> \frac{q-1}{2(2q)^k}.
\]
Hence, the value of \eqref{eq:common-fourier} is greater than $2^{1-k}+\frac{q-1}{2(2q)^k}$ with positive probability. Since the expectation of \eqref{eq:common-fourier} is $2^{1-k}$, it follows that there is some $f$ such that the value of $\eqref{eq:common-fourier}$ is strictly less than $2^{1-k}$, as desired.

\medskip

\noindent (c) Since $k$ is odd, by \eqref{eq:common-fourier} and the convexity of $t \mapsto t^k$, 
\[
\Lambda_{L=0}(f) + \Lambda_{L=0}(1-f) = (\E f)^k + (1-\E f)^k \ge 2^{1-k},
\]
and thus $L = 0$ is common.

On the other hand, setting $\wh f(0) = 1/2$ and $\wh f(r) = -1/(2q)$ for all $r \in \wh{\FF_q^n} \setminus \{0\}$, we see from \eqref{eq:fourier-L} that $\Lambda_{L=0}(f) < (\E f)^k$, so that $L =0$ is not Sidorenko. (Alternatively, we can set $A = \FF_q \setminus \{0\}$ and deduce by  the inclusion-exclusion principle that $\Lambda_{L=0}(1_A) = (1 - 1/q)^k + (-1/q)^k(q-1)$.)
\end{proof}

\begin{proof}[Proof of Theorem~\ref{thm:free-var}]
When $a_1, \dots, a_k$ can be partitioned into pairs each summing to zero, the proof that $L' = 0$ is Sidorenko and common follows from \ref{thm:main}(a) since in this case $\Lambda_{L'=0}(f) = (\E f)^\ell \Lambda_{L=0}(f) \ge (\E f)^{k+\ell}$.

Now consider the case when $k$ is odd. It suffices to exhibit a function $f \colon \FF_q \to [0,1]$ such that $\Lambda_{L'=0}(f) + \Lambda_{L'=0}(f) < 2^{1-k-\ell}$. Consider the function $f \colon \FF_q \to [0,1]$ defined by 
\begin{equation} \label{eq:f-free-var-odd}
f(x) = \begin{cases}
\tfrac12 + c - (q-1)\beta & \text{if } x = 0, \\
\tfrac12 + c + \beta & \text{if } x \ne 0.
\end{cases}
\end{equation}
where $\beta = 1/(2q)$ is a constant and $c >0$ is sufficiently small.
Its Fourier transform is given by
\begin{equation} \label{eq:f-hat-free-var-odd}
\wh f (r) = \begin{cases}
\tfrac12 + c & \text{if } r = 0, \\
-\beta & \text{if } r \ne 0,
\end{cases}
\quad \text{and} \quad 
\wh {(1-f)} (r) = \begin{cases}
\tfrac12 - c & \text{if } r = 0, \\
\beta & \text{if } r \ne 0.
\end{cases}
\end{equation}
From \eqref{eq:fourier-L}, we find that
\[
\Lambda_{L=0}(f) = (\tfrac12 + c)^k + (q-1)(-\beta)^k
\]
and
\[
\Lambda_{L=0}(1-f) = (\tfrac12 - c)^k + (q-1)\beta^k.
\]
Since $\Lambda_{L'=0}(f) = (\E f)^\ell \Lambda_{L=0} (f)$ and $\Lambda_{L'=0}(1-f) = (1- \E f)^\ell \Lambda_{L=0} (1-f)$, and recall that $k$ is odd, we have
\begin{align*}
\Lambda_{L'=0}(f) + \Lambda_{L'=0}(1-f) 
&= \paren{\tfrac12 + c}^{\ell}
\paren{\paren{\tfrac12 + c}^k - (q-1)\beta^k}
+ \paren{\tfrac12 - c}^{\ell}\paren{\paren{\tfrac12 - c}^k + (q-1)\beta^k}
\\
&= \paren{\tfrac12 + c}^{k+\ell} + \paren{\tfrac12 - c}^{k+\ell} 
- (q-1) \beta^k \paren{\paren{\tfrac12 + c}^\ell - \paren{\tfrac12 - c}^\ell}
\\
&= 2^{1-k-\ell} -(q-1)\beta^k 2^{-\ell + 1} \ell c + O(c^2),
\end{align*}
which is less than $2^{- k - \ell + 1}$ as long as $c$ is small enough.

Finally, suppose $k$ is even but $a_1,\dots, a_k$ cannot be partitioned into pairs summing to zero. In the proof of Theorem~\ref{thm:main}(b) we constructed an $f \colon \FF_q \to [0,1]$ with $\E f = 1/2$ such that $\Lambda_{L=0}(f) + \Lambda_{L=0}(f) < 2^{1-k}$. Then this $f$ has $\Lambda_{L'=0}(f) = (\E f)^{\ell} \Lambda_{L=0}(f) = 2^{-\ell} \Lambda_{L=0}(f)$ and likewise $\Lambda_{L'=0} (f) = 2^{-\ell} \Lambda_{L=0}(f)$, and hence $\Lambda_{L'=0}(f) + \Lambda_{L'=0}(f) < 2^{1-k- \ell}$, so that $L'$ is not Sidorenko and not common.
\end{proof}

\begin{proof}[Proof of Theorem~\ref{thm:inhom}]
The equation $L'(\bx) = b$ with $b \ne 0$ has no solutions with all $k$ coordinates of $\bx=(x_1,\ldots,x_k)$ lying in a subspace of $\FF_q^n$ not containing $b$. Thus $L'$ cannot be inhomogeneous-Sidorenko.

Let us define
\[
\Lambda_{L = b}(f) := \E_{\bx = (x_1, \dots, x_k)\in (\FF_q^n)^k : L(\bx) = b}[f(x_1) \cdots f(x_k)].
\]
An extension of \eqref{eq:fourier-L} for inhomogeneous equations gives
\begin{equation}\label{eq:Lambda-inhom-fourier}
\Lambda_{L=b}(f) = \sum_{r \in \wh{\FF_q^n}} \wh f(a_1r) \cdots \wh f(a_kr) r(b).
\end{equation}
Similar to \eqref{eq:common-fourier}, we have
\begin{equation}\label{eq:inhom-common-fourier}
\Lambda_{L=b}(f) + \Lambda_{L=b}(1-f) = (\E f)^k + (1 - \E f)^k + (1+(-1)^k)\sum_{r \in \wh{\FF_q^n} \setminus \{0\}} \wh f(a_1 r) \cdots \wh f(a_k r) r(b).
\end{equation}
Thus when $\ell = 0$ and $k$ is odd, then the above expression is at least $(\E f)^k + (1- \E f)^k \ge 2^{1-k}$ by the convexity of $t \mapsto t^k$, so that $L$ is inhomogeneous-common, and $L' = L$ since $\ell = 0$. 

Now we turn to the negative cases. To show that $L'$ is not inhomogeneous-common, it suffices to exhibit a function $f \colon \FF_q \to [0,1]$ such that $\Lambda_{L'=1}(f) + \Lambda_{L'=1}(1-f) < 2^{1-k-\ell}$. 

When $k$ is even, set $\E f = 1/2$ and $\wh f(r) = 1/(2q)$ for all $r \in \wh{\FF_q^n} \setminus \{0\}$, say, so that $f$ takes values in $[0,1]$, and then the value of \eqref{eq:inhom-common-fourier} is less than $2^{1-k}$ since $\sum_{r \in \wh{\FF_q^n} \setminus \{0\}} r(b) = - 1$, so that $\Lambda_{L=1}(f) + \Lambda_{L=1}(1-f) < 2^{1-k}$. Adding $\ell$ free variables in this case gives $\Lambda_{L=1}(f) + \Lambda_{L=1}(1-f) = (\E f)^\ell \Lambda_{L=1}(f) +  (1- \E f)^\ell\Lambda_{L=1}(1-f) < 2^{1-k-\ell}$.

Finally, suppose $\ell > 0$ and $k$ is odd. Take $f$ as in \eqref{eq:f-free-var-odd} except that now we set $\beta = -1/(2q)$. Then, using \eqref{eq:f-hat-free-var-odd} and \eqref{eq:Lambda-inhom-fourier}, and using that $\sum_{r \in \wh{\FF_q^n} \setminus \{0\}} r(1) = -1$, we have
\[
\Lambda_{L=1}(f) = \paren{\tfrac12 + c}^k - (-\beta)^k
\]
and
\[
\Lambda_{L=1}(1-f) = \paren{\tfrac12 - c}^k - \beta^k.
\]
Since $\Lambda_{L'=1}(f) = (\E f)^\ell \Lambda_{L=1}(f)$ and  $\Lambda_{L'=1}(1-f) = (1-\E f)^\ell \Lambda_{L=1}(1-f)$, 
and recall that $k$ is odd, we have
\begin{align*}
\Lambda_{L'=1}(f) + \Lambda_{L'=1}(1-f) 
&= \paren{\tfrac12 + c}^{\ell}
\paren{\paren{\tfrac12 + c}^k - (-\beta)^k}
+ \paren{\tfrac12 - c}^{\ell}\paren{\paren{\tfrac12 - c}^k - \beta^k}
\\
&= \paren{\tfrac12 + c}^{k+\ell} + \paren{\tfrac12 - c}^{k+\ell} 
+ \beta^k \paren{\paren{\tfrac12 + c}^\ell - \paren{\tfrac12 - c}^\ell}
\\
&= 2^{1-k-\ell} + \beta^k 2^{1-\ell} \ell c + O(c^2),
\end{align*} 
which is less $2^{1-k-\ell}$ if $c > 0$ is small enough (recall that $\beta = -1/(2q)$ and $k$ is odd).
\end{proof}

\section{Concluding Remarks}

In this paper, we characterized which linear equations are common and which are Sidorenko. It is natural to try to characterize which {\it systems} of linear equations satisfy these properties, and Saad and Wolf \cite{SaWo} proved several results on this. To be more precise, a system of linear equations ${\bf x}M={\bf 0}$ in $k$ variables is {\it Sidorenko} if the density of solutions to ${\bf x}M={\bf 0}$ in a set $A$ is at least the $k$-th power of the density of $A$.  Here, we think of ${\bf x}$ as a matrix whose columns are $x_1,\dots,x_k$. A system of linear equations ${\bf x}M={\bf 0}$ in $k$ variables is {\it common} if the density of monochromatic solutions to ${\bf x}M={\bf 0}$ in any two-coloring of $\mathbb{F}_q^n$ is at least $2^{1-k}$. We currently don't have a guess for a characterization of these properties for general linear systems. One reason why this is of interest is that it might lead to a better understanding of the analogous properties for graphs and hypergraphs. 

\begin{question} 
Which systems of linear equations are Sidorenko? Which are common? 
\end{question} 

While the linear homogeneous equation $x_1-2x_2+x_3=0$ giving three-term arithmetic progressions is not Sidorenko, Green \cite{G} introduced a weakening of the Sidorenko property which this equation does satisfy. Green proved that, for each $\epsilon>0$ there is $N(\epsilon)$ such that if $G$ is an abelian group with $|G| \geq N(\epsilon)$ and $A \subseteq G$ has density $\alpha$, then there is a nonzero $d \in G$ such that the density of three-term arithmetic progressions with common difference $d$ is at least $\alpha^3-\epsilon$. That is, while the total density of three-term arithmetic progressions can be much less than given by the random bound, there is a nonzero $d$ for which the density of three-term arithmetic progressions with common difference $d$ is at least almost the random bound. 
The proof uses an arithmetic regularity lemma, and consequently gives a tower-type bound on $N(\epsilon)$. The authors later proved that such a tower-type bound is needed (see \cite{FP,FP2,FPZ}). 

It is natural to try to see whether other systems of linear equations satisfy such a popular differences property. Green and Tao \cite{G07,GT} proved that the linear system giving four-term arithmetic progressions has this property, while Ruzsa \cite{BHKR} proved that for any longer length arithmetic progression, the corresponding linear system does not have this property. 

We say that a single linear homogeneous equation $L=0$ is {\it popular} if, for each $\epsilon>0$ there is $n_L(\epsilon)$ such that if $n \geq n_L(\epsilon)$ and $A \subseteq \mathbb{F}_q^n$ has density $\alpha$, then there are nonzero and distinct $d_1,\ldots,d_{k-1}$ such that the density of solutions to $L=0$ with $x_{i+1}-x_1=d_i$ for $i=1,\ldots,k-1$ is at least $\alpha^k - \epsilon$. We refer to $d_1,\dots,d_{k-1}$ as the {\it popular differences}. A linear homogeneous equation $L=0$ is {\it translation-invariant} if and only if the sum of the coefficients is zero. If the equation $L=0$ is popular, then it must be translation-invariant. Indeed, if $L=0$ is not translation-invariant, then the affine subspace $S$ of codimension one (so density $\alpha=1/q$) consisting of those elements whose first coordinate is one is such that $L=0$ has no solution in $S$. Hence, it follows that $n_{L}(\epsilon)$ does not exist for $\epsilon < 1/q^k$. 

For a translation-invariant linear equation, the reader may compare the above definition of popular differences to the discussion of popular difference for three-term arithmetic progressions, where the popular difference is a single parameter $d=x_2-x_1$ instead of the tuple $(d,2d)=(x_2-x_1,x_3-x_1)$. Indeed, in the general case, fixing $d_1,\dots,d_{k-2}$ would determine a unique choice of $d_{k-1}$ for which there exists a solution to $L({\bf x})=0$ with $x_{i+1}-x_1=d_i$ for $i=1,\ldots,k-1$. However, in the case of three-term arithmetic progressions over abelian groups of odd order, fixing $x_2-x_1=d\ne 0$ ensures that $x_3-x_1=2d\ne 0$, whereas fixing nonzero and distinct $(d_1,\dots,d_{k-2})$ in the general case does not guarantee that $x_k-x_1=d_{k-1}$ is nonzero or distinct from $d_1,\dots,d_{k-2}$. Thus, the inclusion of $d_{k-1}$ in the definition of popular differences does not correspond to a degree of freedom, but only serves to ensure that $x_1,\dots,x_k$ are distinct when $x_{i+1}-x_1=d_i$ for $1\le i\le k-1$. 

Note that if $L=0$ is Sidorenko, then simply by averaging, it is also popular and furthermore, $n_L(\epsilon)$ is bounded above by $O(\log(1/\epsilon))$. Green's theorem shows that the equation $x_1-2x_2+x_3=0$ is popular. More generally, Green's argument in \cite{G} can be extended to show that $L=0$ is popular if and only if it is translation-invariant. Indeed, we showed above that if $L=0$ is not translation-invariant, then it is not popular. If $L=0$ is translation-invariant, then the arithmetic regularity lemma proof of Green's theorem goes as follows. For each subset $A \subseteq \mathbb{F}_q^n$ of density $\alpha$, by a Szemer\'edi type lemma, there is a regular subspace $H$ of bounded codimension, and the counting lemma and Jensen's inequality gives that the density of solutions to $L=0$ with $x_1,\ldots,x_k$ all in the same translate of $H$ is at least almost $\alpha^k$. By throwing out the solutions with $x_1,\ldots,x_k$ not all distinct (which is of smaller order for $n$ sufficiently large) and averaging, we get that there exists nonzero and distinct $d_1,\ldots,d_{k-1} \in H$ for which the density of solutions to $L=0$ with $x_{i+1}-x_1=d_i$ for $1 \leq i \leq k-1$ is at least $\alpha^k-\epsilon$. This proof gives an upper bound on $n_L(\epsilon)$ which is a tower of height $\epsilon^{-O(1)}$.

The first two authors \cite{FP} showed, in the case $L=0$ corresponding to three-term arithmetic progressions, $n_L(\epsilon)$ is in fact bounded above and below by a tower of height $\Theta(\log(1/\epsilon))$. We can directly adapt the proof of the upper bound in \cite{FP} to show that for all translation invariant $L=0$, we have $n_L(\epsilon)$ is bounded above by a tower of height $\Theta(\log(1/\epsilon))$, using a density increment argument with the mean $k$-th power density, defined as $b_k(H)=\E[f_H(x)^k]$, where $f_H(x)$ is the average value of $f$ on the affine translate of $H$ containing $x$. 

The lower bound construction in \cite{FP} heavily depends on the fact that the equation $x_1-2x_2+x_3=0$ is not Sidorenko, as a crucial ingredient of our construction is a model function with relatively low density of three-term arithmetic progressions. As mentioned above, if $L=0$ is Sidorenko, then $n_L(\epsilon)$ is not of tower-type, but in fact, only $O(\log(1/\epsilon))$. 

The converse does not hold. Indeed, we next exhibit a linear homogeneous equation in eight variables which is not Sidorenko but $n_L(\epsilon) = O(\log(1/\epsilon))$. For example,  we may take $L=0$ with  
\begin{equation}
    L(x_1,x_2,x_3,x_4,x_5,x_6,x_7,x_8)=-6x_1+3x_2+x_3+7x_4+2x_5-4x_6-2x_7-x_8. \label{eq:eq-L}
\end{equation}
A {\it Hilbert cube of dimension $t$} is a sequence of $2^t$ numbers $x_1+\sum_{i=1}^t \epsilon_id_i$ with $\epsilon_i \in \{0,1\}$. The Hilbert cubes of dimension $t$ are the solution set to a system ${\bf x}M_t={\bf 0}$ of linear equations in $2^t$ variables. It follows from repeated application of the Cauchy--Schwarz inequality that the system ${\bf x}M_t={\bf 0}$ is Sidorenko (see Example 2.5 in \cite{SaWo}). It follows that if $n > C \log(1/\epsilon)$ for a large enough constant $C>0$, and $A \subseteq \mathbb{F}_q^n$, then, letting $t=3$ and averaging, there is a set $\{d_1,d_2,d_3\}$ such that all eight subsets of this set have nonzero and distinct sums, and the density of Hilbert cubes of dimension three in $A$ with this choice of $d_1,d_2,d_3$ is at least $\E[1_A]^8-\epsilon$. Let $d_4=d_1+d_2$, $d_5=d_1+d_3$, $d_6=d_2+d_3$, and $d_7=d_1+d_2+d_3$. The Hilbert cubes $(x_1,x_1+d_1,x_1+d_2,x_1+d_3,x_1+d_1+d_2,x_1+d_1+d_3,x_1+d_2+d_3,x_1+d_1+d_2+d_3)$ of dimension three with differences $d_1,d_2,d_3$ are precisely the solutions to the example $L=0$ in eight variables above which further satisfy $x_{i+1}-x_1=d_i$ for $i=1,\ldots,7$. Thus there are equations $L=0$ like the one above which are not Sidorenko but $n_{L}(\epsilon) = O(\log(1/\epsilon))$.   

We do not know if there is a translation-invariant linear equation $L=0$ with at least four variables for which $n_L(\epsilon)$ has tower-type growth. It would be further interesting to characterize them. For equations $L=0$ with three variables which are translation-invariant, we can show a tight tower-type lower bound on $n_L(\epsilon)$ by a direct generalization of the argument in \cite{FP}.

\begin{question}
Characterize the growth rate of $n_L(\epsilon)$ for linear equations $L=0$ with at least four variables.
\end{question}

Related to the previous question, it would be interesting to know for which linear equations $L=0$ can we prove $n_L(\epsilon) = O(\log(1/\epsilon))$ through the above technique by finding a linear system which is Sidorenko and whose solution set is a subset of the set of solutions to $L=0$. 

\begin{question}
For which linear equation $L=0$ can we find a Sidorenko linear system of equations ${\bf x}M={\bf 0}$ whose solutions are also solutions of $L=0$?  
\end{question}

Note that the proof of the popular differences theorem using the regularity lemma yields the stronger result that there exists a subspace $H$ of bounded codimension such that the density of solutions to $L=0$ with $x_1,\dots,x_k$ all in the same translate of $H$ is at least $\alpha^k - \epsilon$. 
Motivated by this result and the example \eqref{eq:eq-L} showing that $n_L(\epsilon)$ may not have tower-type growth for equations $L$ which are not Sidorenko, we consider a stronger notion of popular differences, replacing differences by subspaces. Call a translation-invariant linear homogeneous equation $L=0$ in $k$ variables {\it subspace popular} if for each $\epsilon>0$ there exists $\tilde{n}_L(\epsilon)$ such that the following holds. If $A \subseteq \mathbb{F}_q^n$ 
has density $\alpha$, then there is a subspace $H$ of $\mathbb{F}_q^n$ of codimension at most $\tilde{n}_L(\epsilon)$ such that the density of solutions in $A$ to $L(x_1, \dots, x_k)=0$ for which $x_1,\ldots,x_k$ lie in the same translate of $H$ is at least $\alpha^k-\epsilon$. The regularity proof shows that if $L=0$ is  translation-invariant, then $\tilde{n}_L(\epsilon)$ exists and is at most tower-type in $\epsilon^{-1}$. By averaging, 
an upper bound on $\tilde{n}_L(\epsilon)$ easily yields an upper bound on $n_L(\epsilon)$. While the example \eqref{eq:eq-L} given above shows that $n_L(\epsilon)$ may not have tower-type growth even for equations $L=0$ which are not Sidorenko, it does not rule out a tower-type bound for $\tilde{n}_L(\epsilon)$. If $L=0$ is translation-invariant but not Sidorenko, must $\tilde{n}_L(\epsilon)$ have tower-type growth?

\begin{question}\label{H} For each translation-invariant $L=0$, what is the growth rate of $\tilde{n}_L(\epsilon)$?  
\end{question}

\smallskip

\noindent \textbf{Acknowledgments.} 
We thank the referees for helpful comments that improved the exposition of the paper.

\end{document}